\documentclass[12pt]{article}

\usepackage{amsmath,amsthm}
\usepackage{amssymb}
\usepackage{hyperref}
\usepackage{lscape}

\begin{document}

\title{Higher-order identities for balancing numbers}
\author{
Takao Komatsu\\
\small School of Mathematics and Statistics\\[-0.8ex]
\small Wuhan University\\[-0.8ex]
\small Wuhan 430072 China\\[-0.8ex]
\small \texttt{komatsu@whu.edu.cn}\\\\
Prasanta Kumar Ray\\
\small VSS University of Technology, Odisha\\[-0.8ex]
\small Burla-768018, India\\[-0.8ex]
\small \texttt{rayprasanta2008@gmail.com}
}

\date{
}

\maketitle

\def\fl#1{\left\lfloor#1\right\rfloor}
\def\cl#1{\left\lceil#1\right\rceil}
\def\stf#1#2{\left[#1\atop#2\right]}
\def\sts#1#2{\left\{#1\atop#2\right\}}

\newtheorem{theorem}{Theorem}
\newtheorem{Prop}{Proposition}
\newtheorem{Cor}{Corollary}
\newtheorem{Lem}{Lemma}

\begin{abstract}
Let $B_n$ be the $n$-th balancing number. In this paper, we give some explicit expressions of $\sum_{l=0}^{2 r-3}(-1)^l\binom{2 r-3}{l}\sum_{j_1+\cdots+j_r=n-2 l\atop j_1,\dots,j_r\ge 1}B_{j_1}\cdots B_{j_r}$ and $\sum_{j_1+\cdots+j_r=n\atop j_1,\dots,j_r\ge 1}B_{j_1}\cdots B_{j_r}$. 
We also consider the convolution identities with binomial coefficients:  
$$
\sum_{k_1+\cdots+k_r=n\atop k_1,\dots,k_r\ge 1}\binom{n}{k_1,\dots,k_r}B_{k_1}\cdots B_{k_r} 
$$ 
This type can be generalized, so that $B_n$ is a special case of the number $u_n$, where $u_n=a u_{n-1}+b u_{n-2}$ ($n\ge 2$) with $u_0=0$ and $u_1=1$.   
\end{abstract}

\section{Introduction}

Higher-order convolutions for various types of numbers (or polynomials) have been studied, with or without binomial (or multinomial) coefficients, including Bernoulli, Euler, Genocchi, Cauchy, Stirling, and Fibonacci numbers (\cite{AD1,AD2,AD3,Komatsu2015,Komatsu2016,KMP,KS2016}). One typical one is due to Euler, given by
$$
\sum_{k=0}^n\binom{n}{k}\mathcal B_k\mathcal B_{n-k}=-n\mathcal B_{n-1}-(n-1)\mathcal B_n\quad(n\ge 0)\,,
$$
where $\mathcal B_n$ are Bernoulli numbers, defined by
$$
\frac{t}{e^t-1}=\sum_{n=0}^\infty\mathcal B_n\frac{t^n}{n!}\quad(|t|<2\pi)\,.
$$

A positive integer $x$ is called {\it balancing number} if
\begin{equation}
1+2+\cdots+(x-1)=(x+1)+\cdots+(y-1)
\label{def:balancing}
\end{equation}
holds for some integer $y\ge x+2$. The problem of determining all balancing numbers leads to a Pell equation, whose solutions in $x$ can be described by the recurrence $B_n=6B_{n-1}-B_{n-2}$ ($n\ge 2$) with $B_0=0$ and $B_1=1$ (see \cite{behera-panda,Finkelstein}).
One of the most general extensions of balancing numbers is when (\ref{def:balancing}) is being replaced by
\begin{equation}
1^k+2^k+\cdots+(x-1)^k=(x+1)^l+\cdots+(y-1)^l\,,
\label{kpower-balancing}
\end{equation}
where the exponents $k$ and $l$ are given positive integers. In the work of Liptai et al. \cite{LLPS} effective and non-effective finiteness theorems on (\ref{kpower-balancing}) are proved.  In \cite{KS2014} a balancing problem of ordinary binomial coefficients is studied.  Some more results can be seen in \cite{panda2009,PR,ray2012}.

The generating function $f(x)$ of balancing numbers $B_n$ is given by
$$
f(x):=\frac{x}{1-6 x+x^2}=\sum_{n=0}^\infty B_n x^n\,.
$$
Then $f(x)$ satisfies the relation:
\begin{equation}
f(x)^2=\frac{x^2}{1-x^2}f'(x)
\label{f2-fd1}
\end{equation}
or
\begin{equation}
(1-x^2)f(x)^2=x^2 f'(x)\,.
\label{f2-fd1-2}
\end{equation}
The left-hand side of (\ref{f2-fd1-2}) is
\begin{align*}
&(1-x^2)\left(\sum_{u=0}^\infty B_u x^u\right)\left(\sum_{v=0}^\infty B_v x^v\right)\\
&=(1-x^2)\sum_{n=0}^\infty\sum_{j=0}^n B_j B_{n-j}x^n\\
&=\sum_{n=0}^\infty\sum_{j=0}^n B_j B_{n-j}x^n-\sum_{n=2}^\infty\sum_{j=0}^{n-2}B_j B_{n-j-2}x^n\,.
\end{align*}
The right-hand side of (\ref{f2-fd1-2}) is
$$
x^2\sum_{n=1}^\infty n B_n x^{n-1}=\sum_{n=0}^\infty(n-1)B_{n-1}x^n\,.
$$
Comparing the coefficients of both sides, we get
\begin{align*}
(n-1)B_{n-1}&=\sum_{j=0}^n B_j B_{n-j}-\sum_{j=0}^{n-2}B_j B_{n-j-2}\\
&=\sum_{j=1}^{n-1}(B_j B_{n-j}-B_{j-1}B_{n-j-1})\,.
\end{align*}
Here, notice that $B_0=0$.
By changing $n$ by $n+1$, we get the following identity.

\begin{theorem}
For $n\ge 1$, we have
$$
n B_n=\sum_{j=1}^n(B_j B_{n-j+1}-B_{j-1}B_{n-j})\,.
$$
\label{th10}
\end{theorem}

Differentiating both sides of (\ref{f2-fd1}) by $x$ and dividing them by $2$, we obtain
\begin{equation}
f(x)f'(x)=\frac{x}{(1-x^2)^2}f'(x)+\frac{x^2}{2(1-x^2)}f''(x)\,.
\label{ffd-fd1fd2}
\end{equation}
By (\ref{f2-fd1}) and (\ref{ffd-fd1fd2}), we get
\begin{align}
f(x)^3&=\frac{x^2}{1-x^2}f(x)f'(x)\notag\\
&=\frac{x^3}{(1-x^2)^3}f'(x)+\frac{x^4}{2(1-x^2)^2}f''(x)
\label{f3}
\end{align}
or
\begin{equation}
(1-x^2)^3 f(x)^3=x^3 f'(x)+\frac{1}{2}x^4(1-x^2)f''(x)\,.
\label{f3-2}
\end{equation}
The left-hand side of (\ref{f3-2}) is equal to equal to
\begin{align*}
&(1-3 x^2+3 x^4-x^6)\sum_{n=0}^\infty\sum_{j_1+j_2+J_3=n\atop j_1, j_2, j_3\ge 0}B_{j_1}B_{j_2}B_{j_3}x^n\\
&=\sum_{l=0}^3\sum_{n=2 l}^\infty(-1)^l\binom{3}{l}\sum_{j_1+j_2+J_3=n-2 l\atop j_1, j_2, j_3\ge 1}B_{j_1}B_{j_2}B_{j_3}x^n\,.
\end{align*}
The right-hand side of (\ref{f3-2}) is
\begin{align*}
&x^3\sum_{n=1}^\infty n B_n x^{n-1}+\frac{x^4}{2}\sum_{n=2}^\infty n(n-1)B_n x^{n-2}-\frac{x^6}{2}\sum_{n=2}^\infty n(n-1)B_n x^{n-2}\\
&=\sum_{n=2}^\infty\frac{(n-1)(n-2)}{2}B_{n-2}x^n-\sum_{n=4}^\infty\frac{(n-4)(n-5)}{2}B_{n-4}x^n\,.
\end{align*}
Comparing the coefficients of both sides, we get the following result.

\begin{theorem}
For $n\ge 4$, we have
$$
\sum_{l=0}^3(-1)^l\binom{3}{l}\sum_{j_1+j_2+J_3=n-2 l\atop j_1, j_2, j_3\ge 1}B_{j_1}B_{j_2}B_{j_3}
=\binom{n-1}{2}B_{n-2}-\binom{n-4}{2}B_{n-4}\,.
$$
\label{th30}
\end{theorem}

In this paper, we give some explicit expression of a more general case $\sum_{l=0}^{2 r-3}(-1)^l\binom{2 r-3}{l}\sum_{j_1+\cdots+j_r=n-2 l\atop j_1,\dots,j_r\ge 1}B_{j_1}\cdots B_{j_r}$ and  
$\sum_{j_1+\cdots+j_r=n\atop j_1,\dots,j_r\ge 1}B_{j_1}\cdots B_{j_r}$. 
We also consider a different type for more general numbers: 
$$
\sum_{k_1+\cdots+k_r=n\atop k_1,\dots,k_r\ge 1}\binom{n}{k_1,\dots,k_r}u_{k_1}\cdots u_{k_r}\,,  
$$ 
where $u_n=a u_{n-1}+b u_{n-2}$ ($n\ge 2$) with $u_0=0$ and $u_1=1$. When $a=6$ and $b=-1$, this is reduced to the convolution identity for balancing numbers. 
When $a=b=1$, this is reduced to the convolution identity for Fibonacci numbers. The corresponding identities for Luca-balancing numbers are also given.

\section{Main results}

First, as a general case of (\ref{f2-fd1}) and (\ref{f3}), we can have the following.

\begin{Lem}
For $r\ge 2$, we have
\begin{equation}
f(x)^r=\frac{x^{2 r-2}f^{(r-1)}(x)}{(r-1)!(1-x^2)^{r-1}}+\sum_{k=1}^{r-2}\frac{\sum_{j=0}^{k-1}\binom{k}{j}\binom{r-2}{k-j-1}x^{2 r-k+2 j-2}}{k(r-k-2)!(1-x^2)^{r+k-1}}f^{(r-k-1)}(x)\,.
\label{lem-iden}
\end{equation}
\label{lem40}
\end{Lem}

\begin{proof}
The proof is done by induction.
It is trivial to see that the identity holds for $r=2$.
Suppose that the identity holds for some $r$.
Differentiating both sides by $x$, we obtain
\begin{align*}
&r f(x)^{r-1}f'(x)\\
&=\frac{x^{2r-2}f^{(r)}(x)}{(r-1)!(1-x^2)^{r-1}}+\frac{(2r-2)x^{2r-3}f^{(r-1)}(x)}{(r-1)!(1-x^2)^{r}}\\
&\quad +\sum_{k=1}^{r-2}\frac{\sum_{j=0}^{k-1}\binom{k}{j}\binom{r-2}{k-j-1}x^{2r-k-2+2j}}{k(r-k-2)!(1-x^2)^{r+k-1}}f^{(r-k)}(x)\\
&\quad +\sum_{k=1}^{r-2}\frac{\sum_{j=0}^{k-1}(2r-k-2+2j)\binom{k}{j}\binom{r-2}{k-j-1}x^{2r-k-3+2j}}{k(r-k-2)!(1-x^2)^{r+k}}f^{(r-k-1)}(x)\\
&\quad +\sum_{k=1}^{r-2}\frac{\sum_{j=0}^{k-1}(3k-2j)\binom{k}{j}\binom{r-2}{k-j-1}x^{2r-k-1+2j}}{k(r-k-2)!(1-x^2)^{r+k}}f^{(r-k-1)}(x)\\
&=\frac{x^{2r-2}f^{(r)}(x)}{(r-1)!(1-x^2)^{r-1}}+\frac{2 x^{2r-3}f^{(r-1)}(x)}{(r-2)!(1-x^2)^{r}}\\
&\quad +\sum_{k=1}^{r-2}\frac{\sum_{j=0}^{k-1}\binom{k}{j}\binom{r-2}{k-j-1}x^{2r-k-2+2j}}{k(r-k-2)!(1-x^2)^{r+k-1}}f^{(r-k)}(x)\\
&\quad +\sum_{k=2}^{r-1}\frac{\sum_{j=0}^{k-2}(2r-k-1+2j)\binom{k-1}{j}\binom{r-2}{k-j-2}x^{2r-k-2+2j}}{(k-1)(r-k-1)!(1-x^2)^{r+k-1}}f^{(r-k)}(x)\\
&\quad +\sum_{k=2}^{r-1}\frac{\sum_{j=1}^{k-1}(3k-2j-1)\binom{k-1}{j-1}\binom{r-2}{k-j-1}x^{2r-k-2+2j}}{(k-1)(r-k-1)!(1-x^2)^{r+k-1}}f^{(r-k)}(x)\\
&=\frac{x^{2r-2}f^{(r)}(x)}{(r-1)!(1-x^2)^{r-1}}
+r\sum_{k=1}^{r-1}\frac{\sum_{j=0}^{k-1}\binom{k}{j}\binom{r-1}{k-j-1}x^{2r-k-2+2j}}{k(r-k-1)!(1-x^2)^{r+k-1}}f^{(r-k)}(x)\,.
\end{align*}
Here, we used the relations
$$
\frac{2}{(r-2)!}+\frac{1}{(r-3)!}=\frac{r}{(r-2)!}\quad(k=1)
$$
and
\begin{align*}
&\frac{r-k-1}{k}\binom{k}{j}\binom{r-2}{k-j-1}+\frac{2r-k-1+2j}{k-1}\binom{k-1}{j}\binom{r-2}{k-j-2}\\
&\qquad\qquad +\frac{3k-2j-1}{k-1}\binom{k-1}{j-1}\binom{r-2}{k-j-1}\\
&=\frac{r}{k}\binom{k}{j}\binom{r-1}{k-j-1}\quad (k\ge 2)\,.
\end{align*}
Together with (\ref{f2-fd1}), we get
\begin{align*}
f(x)^{r+1}&=\frac{x^2}{1-x^2}f(x)^{r-1}f'(x)\\
&=\frac{x^2}{1-x^2}\left(\frac{x^{2r-2}f^{(r)}(x)}{r!(1-x^2)^{r-1}}
+\sum_{k=1}^{r-1}\frac{\sum_{j=0}^{k-1}\binom{k}{j}\binom{r-1}{k-j-1}x^{2r-k-2+2j}}{k(r-k-1)!(1-x^2)^{r+k-1}}f^{(r-k)}(x)\right)\\
&=\frac{x^{2r}f^{(r)}(x)}{r!(1-x^2)^{r}}
+\sum_{k=1}^{r-1}\frac{\sum_{j=0}^{k-1}\binom{k}{j}\binom{r-1}{k-j-1}x^{2r-k+2j}}{k(r-k-1)!(1-x^2)^{r+k}}f^{(r-k)}(x)\,.
\end{align*}
\end{proof}

In general, we can state the following.

\begin{theorem}
Let $r\ge 2$.  Then for $n\ge 3 r-5$, we have
\begin{multline*}
\sum_{l=0}^{2r-3}(-1)^l\binom{2r-3}{l}\sum_{j_1+\cdots+j_r=n-2l\atop j_1,\dots,j_r\ge 1}B_{j_1}\cdots B_{j_r}\\
=\sum_{k=1}^{r-1}(-1)^{k-1}\frac{n-2k-r+3}{r-1}\binom{n-2k+1}{r-k-1}\binom{n-k-2r+3}{k-1}B_{n-2k-r+3}\,.
\end{multline*}
\label{th40}
\end{theorem}

\begin{proof}
By Lemma \ref{lem40} we get
\begin{multline}
(1-x^2)^{2r-3}f(x)^r
=(1-x^2)^{r-2}\frac{x^{2r-2}f^{(r-1)}(x)}{(r-1)!}\\
+\sum_{k=1}^{r-2}(1-x^2)^{r-k-2}\frac{\sum_{j=0}^{k-1}\binom{k}{j}\binom{r-2}{k-j-1}x^{2r-k-2+2j}}{k(r-k-2)!}f^{(r-k-1)}(x)\,.
\label{eq345}
\end{multline}
Since $B_0=0$, the left-hand side of (\ref{eq345}) is equal to
\begin{align*}
&(1-x^2)^{2r-3}\sum_{n=0}^\infty\sum_{j_1+\cdots+j_r=n\atop j_1,\dots,j_r\ge 0}B_{j_1}\cdots B_{j_r}x^n\\
&=\sum_{l=0}^{2r-3}\sum_{n=2 l}^\infty(-1)^l\binom{2r-3}{l}\sum_{j_1+\cdots+j_r=n-2l\atop j_1,\dots,j_r\ge 1}B_{j_1}\cdots B_{j_r}x^n\,.
\end{align*}
On the other hand,
\begin{align*}
&(1-x^2)^{r-2}\frac{x^{2r-2}f^{(r-1)}(x)}{(r-1)!}\\
&=\sum_{i=0}^{r-2}\binom{r-2}{i}x^{2i}\frac{x^{2r-2}}{(r-1)!}\sum_{n=r-1}^\infty\frac{n!}{(n-r+1)!}B_n x^{n-r+1}\\
&=\frac{1}{(r-1)!}\sum_{i=0}^{r-2}\binom{r-2}{i}\sum_{n=2r+2i-2}^\infty\frac{(n-r-2i+1)!}{(n-2r-2i+2)!}B_{n-r-2i+1}x^n\,.
\end{align*}
For $i=r-2$, we have
\begin{align*}
&\frac{(-1)^{r-2}}{(r-1)!}\sum_{n=4r-6}^\infty\frac{(n-3r+5)!}{(n-4r+6)!}B_{n-3r+5}x^n\\
&=(-1)^{r-2}\sum_{n=3r-5}^\infty\frac{n-3r+5}{r-1}\binom{n-3r+4}{r-2}B_{n-3r+5}x^n\,,
\end{align*}
which yields the term for $k=r-1$ on the right-hand side of the identity in Theorem \ref{th40}.  Notice that
$$
\binom{\gamma'}{\gamma}=0\quad(\gamma'<\gamma)\,.
$$
The second term of the right-hand side of (\ref{eq345}) is
\begin{align*}
&\sum_{k=1}^{r-2}\sum_{i=0}^{r-k-2}(-1)^i\binom{r-k-2}{i}x^{2i}\frac{1}{k(r-k-2)!}\sum_{j=0}^{k-1}
\binom{k}{j}\binom{r-2}{k-j-1}x^{2r-k-2+2j}\\
&\qquad\times\sum_{n=r-k-1}^\infty\frac{n!}{(n-r+k+1)!}B_n x^{n-r+k+1}\\
&=\sum_{i=0}^{r-3}\sum_{j=0}^{r-i-3}\sum_{k=j}^{r-i-3}\frac{(-1)^i}{(k+1)(r-k-3)!(n-2r+k-2i-2j+3)!}\binom{r-k-3}{i}\\
&\quad\times\binom{k+1}{j}\binom{r-2}{k-j}\sum_{n=2r+2i+2j-k-3}(n-r-2i-2j+1)!B_{n-r-2i-2j+1}x^n\\
&=\sum_{i=0}^{r-3}\sum_{\kappa=i+1}^{r-2}\sum_{k=\kappa-i-1}^{r-i-3}\frac{(-1)^i}{(k+1)(r-k-3)!(n-2r+k-2\kappa+5)!}\binom{r-k-3}{i}\\
&\quad\times\binom{k+1}{\kappa-i-1}\binom{r-2}{k-\kappa+i+1}\sum_{n=2r+2\kappa-k-5}(n-r-2\kappa+3)!B_{n-r-2\kappa+3}x^n\,.
\end{align*}
Together with the first term of the right-hand side of (\ref{eq345}) we can prove that
\begin{align}
&\frac{(-1)^{k-1}}{(r-1)!}\binom{r-2}{k-1}\frac{(n-r-2k+3)!}{(n-2r-2k+4)!}\notag\\
&\quad +\sum_{i=0}^{k-1}\sum_{l=k-i-1}^{r-i-3}\frac{(-1)^{k-1}}{(l+1)(r-l-3)!(n-2r+l-2k+5)!}\notag\\
&\qquad\times\binom{r-l-3}{i}\binom{l+1}{k-i-1}\binom{r-2}{l-k+i+1}(n-r-2k+3)!\notag\\
&=(-1)^{k-1}\frac{n-2k-r+3}{r-1}\binom{n-2k+1}{r-k-1}\binom{n-k-2r+3}{k-1}\,.
\label{mainsub}
\end{align}
Then the proof is done.
\end{proof}

\section{Examples}

When $r=2$ and $r=3$,  Theorem \ref{th40} is reduced to Theorem \ref{th10} and Theorem \ref{th30}, respectively.
When $r=4,5,6$ in Theorem \ref{th40}, we get the following Corollaries as examples.

\begin{Cor}
For $n\ge 7$, we have
\begin{multline*}
\sum_{l=0}^5(-1)^l\binom{5}{l}\sum_{j_1+j_2+j_3+j_4=n-2l\atop j_1,j_2,j_3,j_4\ge 1}B_{j_1}B_{j_2}B_{j_3}B_{j_4}\\
=\binom{n-1}{3}B_{n-3}-\frac{(n-3)(n-5)(n-7)}{3}B_{n-5}+\binom{n-7}{3}B_{n-7}\,.\end{multline*}
\label{th480}
\end{Cor}

\begin{Cor}
For $n\ge 10$, we have
\begin{multline*}
\sum_{l=0}^7(-1)^l\binom{7}{l}\sum_{j_1+j_2+j_3+j_4+j_5=n-2l\atop j_1,j_2,j_3,j_4,j_5\ge 1}B_{j_1}B_{j_2}B_{j_3}B_{j_4}B_{j_5}\\
=\binom{n-1}{4}B_{n-4}-\frac{(n-3)(n-4)(n-6)(n-9)}{8}B_{n-6}\\
+\frac{(n-5)(n-8)(n-10)(n-11)}{8}B_{n-6}-\binom{n-10}{4}B_{n-10}\,.\end{multline*}
\label{th580}
\end{Cor}

\begin{Cor}
For $n\ge 13$, we have
\begin{multline*}
\sum_{l=0}^9(-1)^l\binom{9}{l}\sum_{j_1+\cdots+j_6=n-2l\atop j_1,\dots,j_6\ge 1}B_{j_1}\cdots B_{j_6}\\
=\binom{n-1}{5}B_{n-5}-\frac{(n-3)(n-4)(n-5)(n-7)(n-11)}{30}B_{n-7}\\
+\frac{(n-5)(n-6)(n-9)(n-12)(n-13)}{20}B_{n-9}\\
-\frac{(n-7)(n-11)(n-13)(n-14)(n-15)}{30}B_{n-11}+\binom{n-13}{5}B_{n-13}\,.\end{multline*}
\label{th680}
\end{Cor}

\section{Another result}

In this section, we shall give an expression of $\sum_{j_1+\cdots+j_r=n\atop j_1,\dots,j_r\ge 1}B_{j_1}\cdots B_{j_r}$.

The left-hand side of (\ref{f2-fd1}) is  equal to
$$
\left(\sum_{u=0}^\infty B_u x^u\right)\left(\sum_{v=0}^\infty B_v x^v\right)=\sum_{n=0}^\infty\sum_{j=0}^n B_j B_{n-j}x^n\,.
$$
The right-hand side of (\ref{f2-fd1}) is  equal to
\begin{align*}
&x^2\left(\sum_{j=0}^\infty x^{2 j}\right)\left(\sum_{m=1}^\infty m B_m x^{m-1}\right)\\
&=x\left(\sum_{j=0}^\infty\frac{1+(-1)^j}{2}x^j\right)\left(\sum_{m=1}^\infty m B_m x^m\right)\\
&=x\sum_{n=0}^\infty\left(\sum_{m=0}^n\frac{1+(-1)^{m}}{2}(n-m)B_{n-m}\right)x^n\\
&=\sum_{n=1}^\infty\left(\sum_{m=0}^{n-1}\frac{1+(-1)^{m}}{2}(n-m-1)B_{n-m-1}\right)x^n\\
&=\sum_{n=1}^\infty\left(\sum_{m=0}^{\fl{\frac{n-1}{2}}}(n-2 m-1)B_{n-2 m-1}\right)x^n\,.
\end{align*}
Comparing the coefficients of both sides, we have the following.

\begin{theorem}
For $n\ge 2$, we have
$$
\sum_{j=1}^{n-1}B_j B_{n-j}=\sum_{m=0}^{\fl{\frac{n-1}{2}}}(n-2 m-1)B_{n-2 m-1}\,.
$$
\label{th50}
\end{theorem}

In general, we have the following.

\begin{theorem}
For $n\ge r\ge 2$, we have
\begin{multline*}
\sum_{j_1+\cdots+j_r=n\atop j_1,\dots,j_r\ge 1}B_{j_1}\cdots B_{j_r}\\
=\sum_{m=0}^{\fl{\frac{n-r+1}{2}}}\binom{n-m-1}{r-2}\binom{m+r-2}{r-2}\frac{n-2 m-r+1}{r-1}B_{n-2 m-r+1}\,.
\end{multline*}
\label{th60}
\end{theorem}

\begin{proof}
The left-hand side of (\ref{lem-iden}) in Lemma \ref{lem40} is equal to
$$
\sum_{n=0}^\infty\sum_{j_1+\cdots+j_r=n\atop j_1,\dots,j_r\ge 1}B_{j_1}\cdots B_{j_r}x^n\,.
$$

The first term on the right-hand side of (\ref{lem-iden}) in Lemma \ref{lem40} is equal to
\begin{align*}
&\frac{x^{2r-2}f^{(r-1)}(x)}{(r-1)!(1-x^2)^{r-1}}\\
&=\frac{x^{2r-2}}{(r-1)!}\sum_{i=0}^\infty\binom{i+r-2}{r-2}x^{2i}\sum_{m=0}^\infty\frac{(m+r-1)!}{m!}B_{m+r-1}x^m\\
&=\frac{x^{2r-2}}{(r-1)!}\sum_{k=0}^\infty\frac{1}{(r-2)!2^{r-2}}\frac{(k+2r-4)!!}{k!!}\frac{1+(-1)^k}{2}x^k\\
&\quad\times \sum_{m=0}^\infty\frac{(m+r-1)!}{m!}B_{m+r-1}x^m\\
&=\frac{x^{2r-2}}{(r-1)!}\sum_{n=0}^\infty\sum_{m=0}^n\frac{1}{(r-2)!2^{r-2}}\frac{(n-m+2r-4)!!}{(n-m)!!}\\
&\qquad \times\frac{1+(-1)^{n-m}}{2}\frac{(m+r-1)!}{m!}B_{m+r-1}x^n\\
&=\frac{1}{(r-1)!(r-2)!2^{r-2}}\sum_{n=2r-2}^\infty\sum_{m=0}^{n-2r+2}\frac{(n-m-2)!!}{(n-m-2r+2)!!}\\
&\qquad \times\frac{1+(-1)^{n-m}}{2}\frac{(m+r-1)!}{m!}B_{m+r-1}x^n\\
&=\frac{1}{(r-1)!(r-2)!2^{r-2}}\sum_{n=2r-2}^\infty\sum_{m=r-1}^{n-r+1}\frac{(n-m+r-3)!!}{(n-m-r+1)!!}\\
&\qquad \times\frac{1+(-1)^{n-m-r+1}}{2}\frac{m!}{(m-r+1)!}B_m x^n\,.
\end{align*}
Concerning the second term, we have
\begin{align*}
&\frac{\sum_{j=0}^{k-1}\binom{k}{j}\binom{r-2}{k-j-1}x^{2r-k-2+2j}}{(1-x^2)^{r+k-1}}f^{(r-k-1)}(x)\\
&=\sum_{j=0}^{k-1}\binom{k}{j}\binom{r-2}{k-j-1}x^{2r-k-2+2j}
\sum_{i=0}^\infty\binom{i+r+k-2}{r+k-2}x^{2i}\\
&\qquad \times\sum_{m=0}^\infty\frac{(m+r-k-1)!}{m!}B_{m+r-k-1}x^m\\
&=\sum_{j=0}^{k-1}\binom{k}{j}\binom{r-2}{k-j-1}x^{2r-k-2+2j}
\sum_{l=0}^\infty\frac{1}{(r+k-2)!2^{r+k-2}}\\
&\qquad \times\frac{(l+2r+2k-4)!!}{l!!}\frac{1+(-1)^l}{2}x^l\sum_{m=0}^\infty\frac{(m+r-k-1)!}{m!}B_{m+r-k-1}x^m\\
&=\sum_{j=0}^{k-1}\binom{k}{j}\binom{r-2}{k-j-1}x^{2r-k-2+2j}
\sum_{n=0}^\infty\sum_{m=0}^n\frac{1}{(r+k-2)!2^{r+k-2}}\\
&\qquad \times\frac{(n-m+2r+2k-4)!!}{(n-m)!!}\frac{1+(-1)^{n-m}}{2}\frac{(m+r-k-1)!}{m!}B_{m+r-k-1}x^n\\
&=\frac{1}{(r+k-2)!2^{r+k-2}}\sum_{j=0}^{k-1}\binom{k}{j}\binom{r-2}{k-j-1}\sum_{n=2r-k-2+2j}^\infty\sum_{m=0}^{n-2r+k+2-2j}\\
&\quad \frac{(n-m+3k-2-2j)!!}{(n-m-2r+k+2-2j)!!}\frac{1+(-1)^{n-m+k}}{2}\frac{(m+r-k-1)!}{m!}B_{m+r-k-1}x^n\,.
\end{align*}
Since
$$
\frac{(n-m+r+2k-3-2j)!!}{(n-m-r+k+3-2j)!!}=0
$$
if $m=n-2r+k+2-2j$ ($j=1,2,\dots,k-2$),
this is equal to
\begin{align*}
&\frac{1}{(r+k-2)!2^{r+k-2}}\sum_{j=0}^{k-1}\binom{k}{j}\binom{r-2}{k-j-1}\sum_{n=2r-k-2}^\infty\sum_{m=0}^{n-2r+k+2}\\
&\times\frac{(n-m+3k-2-2j)!!}{(n-m+k)!!}\frac{1+(-1)^{n-m+k}}{2}\frac{(m+r-k-1)!}{m!}B_{m+r-k-1}x^n\\
&=\frac{1}{(r+k-2)!2^{r+k-2}}\sum_{n=2r-k-2}^\infty\sum_{m=r-k-1}^{n-r+1}\sum_{j=0}^{k-1}\binom{k}{j}\binom{r-2}{k-j-1}\\
&\quad \times\frac{(n-m+r+2k-3-2j)!!}{(n-m-r+k+3-2j)!!}\frac{1+(-1)^{n-m-r+1}}{2}\frac{m!}{(m-r+k+1)!}B_m x^n\,.
\end{align*}
This is also equal to
\begin{align*}
&\frac{1}{(r+k-2)!2^{r+k-2}}\sum_{n=2r-k-2}^\infty\sum_{m=r-k-1}^{n-r+1}\sum_{j=0}^{k-1}\binom{k}{j}\binom{r-2}{k-j-1}\\
&\quad \times\frac{(n-m+r+2k-3-2j)!!}{(n-m-r+k+3-2j)!!}\frac{1+(-1)^{n-m-r+1}}{2}\frac{m!}{(m-r+k+1)!}B_m x^n\\
&=\frac{1}{(r+k-2)!2^{r+k-2}}\sum_{n=2r-k-2}^\infty\sum_{m=r-k-1}^{n-r+1}\frac{(n-m+r-1)!!}{(n-m-r+1)!!}\binom{r+k-2}{k-1}\\
&\quad \times\frac{(n-m+r-3)!!}{(n-m+r-2k-1)!!}\frac{1+(-1)^{n-m-r+1}}{2}\frac{m!}{(m-r+k+1)!}B_m x^n\,.
\end{align*}
Therefore, the right-hand side of the relation in Theorem \ref{th50} is
\begin{align*}
&\frac{1}{(r-1)!(r-2)!2^{r-2}}\sum_{n=2r-2}^\infty\sum_{m=r-1}^{n-r+1}\frac{(n-m+r-3)!!}{(n-m-r+1)!!}\frac{1+(-1)^{n-m-r+1}}{2}\\
&\qquad\times\frac{m!}{(m-r+1)!}B_m x^n\\
&\quad +\sum_{k=1}^{r-1}\frac{1}{k(r-k-2)!}\frac{1}{(r+k-2)!2^{r+k-2}}\sum_{n=2r-k-2}^\infty\\
&\qquad \sum_{m=r-k-1}^{n-r+1}\frac{(n-m+r-1)!!}{(n-m-r+1)!!}\binom{r+k-2}{k-1}\frac{(n-m+r-3)!!}{(n-m+r-2k-1)!!}\\
&\qquad\qquad \times\frac{1+(-1)^{n-m-r+1}}{2}\frac{m!}{(m-r+k+1)!}B_m x^n\\
&=\frac{1}{(r-1)!(r-2)!2^{r-2}}\sum_{n=2r-2}^\infty\sum_{m=r-1}^{n-r+1}\frac{(n-m+r-3)!!}{(n-m-r+1)!!}\frac{1+(-1)^{n-m-r+1}}{2}\\
&\qquad\times\frac{m!}{(m-r+1)!}B_m x^n\\
&\quad +\sum_{n=r-1}^\infty\frac{1}{(r-1)!2^{r-2}}\sum_{m=1}^{r-2}\sum_{k=r-m-1}^{r-2}\frac{1}{k!(r-k-2)!2^k}\frac{(n-m+r-1)!!}{(n-m-r+1)!!}\\
&\qquad \times\frac{(n-m+r-3)!!}{(n-m+r-2k-1)!!}\frac{1+(-1)^{n-m-r+1}}{2}\frac{m!}{(m-r+k+1)!}B_m x^n\\
&\quad +\sum_{n=r-1}^\infty\frac{1}{(r-1)!2^{r-2}}\sum_{m=r-1}^{n-r+1}\sum_{k=1}^{r-2}\frac{1}{k!(r-k-2)!2^k}\frac{(n-m+r-1)!!}{(n-m-r+1)!!}\\
&\qquad \times\frac{(n-m+r-3)!!}{(n-m+r-2k-1)!!}\frac{1+(-1)^{n-m-r+1}}{2}\frac{m!}{(m-r+k+1)!}B_m x^n\,.
\end{align*}
Since for $1\le m\le r-2$ we have
\begin{align*}
&\frac{1}{(r-1)!2^{r-2}}\sum_{k=r-m-1}^{r-2}\frac{1}{k!(r-k-2)!2^k}\frac{(n-m+r-1)!!}{(n-m-r+1)!!}\\
&\qquad \times\frac{(n-m+r-3)!!}{(n-m+r-2k-1)!!}\frac{m!}{(m-r+k+1)!}\\
&=\frac{1}{(r-1)!(r-2)!2^{2r-4}}\frac{(n+m+r-3)!!}{(n+m-r+1)!!}\frac{(n-m+r-3)!!}{(n-m-r+1)!!}m
\end{align*}
and for $r-1\le m\le n-r+1$ we have
\begin{align*}
&\frac{1}{(r-1)!(r-2)!2^{r-2}}\frac{(n-m+r-3)!!}{(n-m-r+1)!!}\frac{m!}{(m-r+1)!}\\
&\quad +\frac{1}{(r-1)!2^{r-2}}\sum_{k=r-m-1}^{r-2}\frac{1}{k!(r-k-2)!2^k}\frac{(n-m+r-1)!!}{(n-m-r+1)!!}\\
&\qquad \times\frac{(n-m+r-3)!!}{(n-m+r-2k-1)!!}\frac{m!}{(m-r+k+1)!}\\
&=\frac{1}{(r-1)!(r-2)!2^{2r-4}}\frac{(n+m+r-3)!!}{(n+m-r+1)!!}\frac{(n-m+r-3)!!}{(n-m-r+1)!!}m\,,
\end{align*}
By comparing the coefficients, we have
\begin{align*}
&\sum_{j_1+\cdots+j_r=n\atop j_1,\dots,j_r\ge 1}B_{j_1}\cdots B_{j_r}\\
&=\frac{1}{(r-1)!(r-2)!2^{2r-4}}\\
&\quad\times\sum_{m=0}^{n-r+1}\frac{(n+m+r-3)!!(n-m+r-3)!!}{(n+m-r+1)!!(n-m-r+1)!!}\frac{1+(-1)^{n-m-r+1}}{2}m F_{m}\\
&=\frac{1}{(r-1)!(r-2)!2^{2r-4}}\\
&\quad\times\sum_{m=0}^{n-r+1}\frac{(2 n-m-2)!!(m+2 r-4)!!}{(2 n-m-2 r+2)!!m!!}(n-m-r+1)B_{n-m-r+1}\\
&=\frac{1}{(r-1)!(r-2)!2^{2r-4}}\\
&\quad\times\sum_{m=0}^{\fl{\frac{n-r+1}{2}}}\frac{(2 n-2 m-2)!!(2 m+2 r-4)!!}{(2 n-2 m-2 r+2)!!(2 m)!!}(n-2 m-r+1)B_{n-2 m-r+1}\\
&=\sum_{m=0}^{\fl{\frac{n-r+1}{2}}}\binom{n-m-1}{r-2}\binom{m+r-2}{r-2}\frac{n-2 m-r+1}{r-1}B_{n-2 m-r+1}\,.
\end{align*}
\end{proof}

\section{Some other generating functions}

Another kinds of the generating functions of balancing numbers and Lucas-balancing numbers are given by
$$
b(t):=\frac{e^{\alpha t}-e^{\beta t}}{4\sqrt{2}}=\sum_{n=0}^\infty B_n\frac{t^n}{n!}
$$
and
$$
c(t):=\frac{e^{\alpha t}+e^{\beta t}}{2}=\sum_{n=0}^\infty C_n\frac{t^n}{n!}
$$
because they satisfy the differential equation $y''-6 y'+y=0$.

Since $b'(t)=3 b(t)+c(t)$ and $c'(t)=8 b(t)+3 c(t)$,  we have for $n\ge 0$
$$
B_{n+1}=3 B_n+C_n
$$
and
$$
C_{n+1}=8 B_n+3 C_n\,.
$$

Since
$$
c(t)^2=\frac{e^{2\alpha t}+e^{2\beta t}}{4}+\frac{e^{6 t}}{2}\,,
$$
we have
$$
\sum_{n=0}^\infty\sum_{k=0}^n\binom{n}{k}C_k C_{n-k}=\frac{1}{2}\sum_{n=0}^\infty C_n\frac{(2 t)^n}{n!}+\frac{1}{2}\sum_{n=0}^\infty\frac{(6 t)^n}{n!}\,,
$$
yielding
$$
\sum_{k=0}^n\binom{n}{k}C_k C_{n-k}=\frac{2^n C_n+6^n}{2}\quad(n\ge 0)\,.
$$
Similarly, by
$$
b(t)^2=\frac{e^{2\alpha t}+e^{2\beta t}}{32}-\frac{e^{6 t}}{16}\,,
$$
we have
$$
\sum_{k=0}^n\binom{n}{k}B_k B_{n-k}=\frac{2^n C_n-6^n}{16}\quad(n\ge 0)\,.
$$

Since by $2\alpha+\beta=\alpha+6$ and $\alpha+2\beta=\beta+6$
\begin{align*}
b(t)^3&=\frac{e^{3\alpha t}-e^{3\beta t}}{(4\sqrt{2})^3}-\frac{3(e^{(2\alpha+\beta)t}-e^{(\alpha+2\beta)t})}{(4\sqrt{2})^3}\\
&=\frac{1}{32}\frac{e^{3\alpha t}-e^{3\beta t}}{4\sqrt{2}}-\frac{3}{32}e^{6t}\frac{e^{\alpha t}-e^{\beta t}}{4\sqrt{2}}\\
&=\frac{1}{32}\sum_{n=0}^\infty B_n\frac{(3 t)^n}{n!}-\frac{3}{32}\sum_{i=0}^\infty\frac{(6 t)^i}{i!}\sum_{k=0}^\infty B_k\frac{t^k}{k!}\\
&=\frac{1}{32}\sum_{n=0}^\infty 3^n B_n\frac{t^n}{n!}-\frac{3}{32}\sum_{n=0}^\infty\sum_{k=0}^n\binom{n}{k}6^{n-k}B_k\frac{t^n}{n!}\,,
\end{align*}
we have
$$
\sum_{k_1+k_2+k_3=n\atop k_1,k_2,k_3\ge 1}\binom{n}{k_1,k_2,k_3}B_{k_1}B_{k_2}B_{k_3}=\frac{1}{32}\left(3^n B_n-3\sum_{k=0}^n\binom{n}{k}6^{n-k}B_k\right)\,.
$$
Notice that $B_0=0$.
Similarly, we can obtain that
$$
\sum_{k_1+k_2+k_3=n\atop k_1,k_2,k_3\ge 0}\binom{n}{k_1,k_2,k_3}C_{k_1}C_{k_2}C_{k_3}=\frac{1}{4}\left(3^n C_n+3\sum_{k=0}^n\binom{n}{k}6^{n-k}C_k\right)\,.
$$

Let $r\ge 1$.   If $r$ is odd, then
\begin{align*}
b(t)^r&=\left(\frac{e^{\alpha t}-e^{\beta t}}{4\sqrt{2}}\right)^r\\
&=\frac{1}{(4\sqrt{2})^r}\sum_{j=0}^{\frac{r-1}{2}}(-1)^j\binom{r}{j}(e^{((r-j)\alpha+j\beta)t}-e^{(j\alpha+(r-j)\beta)t})\\
&=\frac{1}{(4\sqrt{2})^r}\sum_{j=0}^{\frac{r-1}{2}}(-1)^j\binom{r}{j}e^{6j t}(e^{(r-2j)\alpha t}-e^{(r-2j)\beta t})\\
&=\frac{1}{(4\sqrt{2})^{r-1}}\sum_{j=0}^{\frac{r-1}{2}}(-1)^j\binom{r}{j}\sum_{i=0}^\infty\frac{(6 j t)^i}{i!}\sum_{k=0}^\infty B_k\frac{\bigl((r-2 j)t\bigr)^k}{k!}\\
&=\frac{1}{(4\sqrt{2})^{r-1}}\sum_{j=0}^{\frac{r-1}{2}}(-1)^j\binom{r}{j}\sum_{n=0}^\infty\sum_{k=0}^n\binom{n}{k}(6 j)^{n-k}(r-2 j)^k B_k\frac{t^n}{n!}\,.
\end{align*}
Therefore, we get
\begin{multline*}
\sum_{k_1+\cdots+k_r=n\atop k_1,\dots,k_r\ge 1}\binom{n}{k_1,\dots,k_r}B_{k_1}\cdots B_{k_r}\\
=\frac{1}{(4\sqrt{2})^{r-1}}\sum_{j=0}^{\frac{r-1}{2}}(-1)^j\binom{r}{j}\sum_{k=0}^n\binom{n}{k}(6 j)^{n-k}(r-2 j)^k B_k\,.
\end{multline*}
Similarly,  we get
\begin{multline*}
\sum_{k_1+\cdots+k_r=n\atop k_1,\dots,k_r\ge 0}\binom{n}{k_1,\dots,k_r}C_{k_1}\cdots C_{k_r}\\
=\frac{1}{2^{r-1}}\sum_{j=0}^{\frac{r-1}{2}}\binom{r}{j}\sum_{k=0}^n\binom{n}{k}(6 j)^{n-k}(r-2 j)^k C_k\,.
\end{multline*}
If $r$ is even, then
\begin{align*}
b(t)^r
&=\frac{1}{(4\sqrt{2})^r}\left(\sum_{j=0}^{\frac{r}{2}-1}(-1)^j\binom{r}{j}(e^{((r-j)\alpha+j\beta)t}+e^{(j\alpha+(r-j)\beta)t})\right.\\
&\quad \left.+(-1)^{\frac{r}{2}}\binom{r}{\frac{r}{2}}e^{(\frac{r}{2}\alpha+\frac{r}{2}\beta)t}\right)\\
&=\frac{1}{(4\sqrt{2})^r}\left(\sum_{j=0}^{\frac{r}{2}-1}(-1)^j\binom{r}{j}e^{6j t}(e^{(r-2j)\alpha t}+e^{(r-2j)\beta t})\right.\\
&\quad \left.+(-1)^{\frac{r}{2}}\binom{r}{\frac{r}{2}}e^{3 r t}\right)\\
&=\frac{1}{(4\sqrt{2})^{r}}\left(\sum_{j=0}^{\frac{r}{2}-1}(-1)^j\binom{r}{j}\sum_{i=0}^\infty\frac{(6 j t)^i}{i!}\cdot 2\sum_{k=0}^\infty C_k\frac{\bigl((r-2 j)t\bigr)^k}{k!}\right.\\
&\quad\left. +(-1)^{\frac{r}{2}}\binom{r}{\frac{r}{2}}\sum_{n=0}^\infty\frac{(3 r)^n}{n!}t^n\right)\\
&=\frac{1}{(4\sqrt{2})^{r}}\left(2\sum_{j=0}^{\frac{r}{2}-1}(-1)^j\binom{r}{j}\sum_{n=0}^\infty\sum_{k=0}^n\binom{n}{k}(6 j)^{n-k}(r-2 j)^k C_k\frac{t^n}{n!}\right.\\
&\quad\left. +(-1)^{\frac{r}{2}}\binom{r}{\frac{r}{2}}\sum_{n=0}^\infty(3 r)^n\frac{t^n}{n!}\right)\,.
\end{align*}
Therefore, we get
\begin{multline*}
\sum_{k_1+\cdots+k_r=n\atop k_1,\dots,k_r\ge 1}\binom{n}{k_1,\dots,k_r}B_{k_1}\cdots B_{k_r}\\
=\frac{1}{(4\sqrt{2})^{r}}\left(2\sum_{j=0}^{\frac{r}{2}-1}(-1)^j\binom{r}{j}\sum_{k=0}^n\binom{n}{k}(6 j)^{n-k}(r-2 j)^k C_k\right.\\
\quad \left.
+(-1)^{\frac{r}{2}}\binom{r}{\frac{r}{2}}(3 r)^n\right)\,.
\end{multline*}
Similarly,  we get
\begin{multline*}
\sum_{k_1+\cdots+k_r=n\atop k_1,\dots,k_r\ge 0}\binom{n}{k_1,\dots,k_r}C_{k_1}\cdots C_{k_r}\\
=\frac{1}{2^{r}}\left(2\sum_{j=0}^{\frac{r}{2}-1}\binom{r}{j}\sum_{k=0}^n\binom{n}{k}(6 j)^{n-k}(r-2 j)^k C_k
+\binom{r}{\frac{r}{2}}(3 r)^n\right)\,.
\end{multline*}

\section{More general cases}

Let $\{u_n\}_{n\ge 0}$ and $\{v_n\}_{n\ge 0}$ be integer sequences, satisfying the same recurrence relation: $u_n=a u_{n-1}+b u_{n-2}$ ($n\ge 2$) and $v_n=a v_{n-1}+b v_{n-2}$ ($n\ge 2$) with initial values $u_0$, $u_1$, $v_0$ and $v_1$.
If the general terms are given by
$$
u_n=\frac{\alpha^n-\beta^n}{\alpha-\beta}\quad\hbox{and}\quad v_n=\alpha^n+\beta^n\quad(n\ge 0)
$$
where
$$
\alpha=\frac{a+\sqrt{a^2+4 b}}{2}\quad\hbox{and}\quad \beta=\frac{a-\sqrt{a^2+4 b}}{2}\,,
$$
we can set $u_0=0$, $u_1=1$, $v=0=2$ and $v_1=a$.
Then, the generating functions of $u_n$ and $v_n$ are given by
$$
u(t):=\frac{e^{\alpha t}-e^{\beta t}}{\sqrt{a^2+4 b}}=\sum_{n=0}^\infty u_n\frac{t^n}{n!}\quad\hbox{and}\quad
v(t):=e^{\alpha t}+e^{\beta t}=\sum_{n=0}^\infty v_n\frac{t^n}{n!}
$$
respectively, because they satisfy the differential equation $y''-a y'-b y=0$.

Our main results can be stated as follows.

\begin{theorem}
If $r$ is odd, then
\begin{multline}
\sum_{k_1+\cdots+k_r=n\atop k_1,\dots,k_r\ge 1}\binom{n}{k_1,\dots,k_r}u_{k_1}\cdots u_{k_r}\\
=\frac{1}{(\sqrt{a^2+4 b})^{r-1}}\sum_{k=0}^n\binom{n}{k}\sum_{j=0}^{\frac{r-1}{2}}(-1)^j\binom{r}{j}(a j)^{n-k}(r-2 j)^k u_k
\label{u:odd}
\end{multline}
and
\begin{multline}
\sum_{k_1+\cdots+k_r=n\atop k_1,\dots,k_r\ge 0}\binom{n}{k_1,\dots,k_r}v_{k_1}\cdots v_{k_r}\\
=\sum_{k=0}^n\binom{n}{k}\sum_{j=0}^{\frac{r-1}{2}}\binom{r}{j}(a j)^{n-k}(r-2 j)^k v_k\,.
\label{v:odd}
\end{multline}
If $r$ is even, then
\begin{multline}
\sum_{k_1+\cdots+k_r=n\atop k_1,\dots,k_r\ge 1}\binom{n}{k_1,\dots,k_r}u_{k_1}\cdots u_{k_r}\\
=\frac{1}{(\sqrt{a^2+4 b})^{r}}\left(\sum_{k=0}^n\binom{n}{k}\sum_{j=0}^{\frac{r}{2}-1}(-1)^j\binom{r}{j}(a j)^{n-k}(r-2 j)^k v_k\right.\\
\quad \left.
+(-1)^{\frac{r}{2}}\binom{r}{\frac{r}{2}}\left(\frac{a r}{2}\right)^n\right)
\label{u:even}
\end{multline}
and
\begin{multline}
\sum_{k_1+\cdots+k_r=n\atop k_1,\dots,k_r\ge 0}\binom{n}{k_1,\dots,k_r}v_{k_1}\cdots v_{k_r}\\
=\sum_{k=0}^n\binom{n}{k}\sum_{j=0}^{\frac{r}{2}-1}\binom{r}{j}(a j)^{n-k}(r-2 j)^k v_k
+\binom{r}{\frac{r}{2}}\left(\frac{a r}{2}\right)^n\,.
\label{v:even}
\end{multline}
\label{th-uv}
\end{theorem}

\begin{proof}
If $r$ is odd, then by $\alpha+\beta=a$
\begin{align*}
u(t)^r&=\left(\frac{e^{\alpha t}-e^{\beta t}}{\sqrt{a^2+4 b}}\right)^r\\
&=\frac{1}{(\sqrt{a^2+4 b})^r}\sum_{j=0}^{\frac{r-1}{2}}(-1)^j\binom{r}{j}(e^{((r-j)\alpha+j\beta)t}-e^{(j\alpha+(r-j)\beta)t})\\
&=\frac{1}{(\sqrt{a^2+4 b})^r}\sum_{j=0}^{\frac{r-1}{2}}(-1)^j\binom{r}{j}e^{a j t}(e^{(r-2j)\alpha t}-e^{(r-2j)\beta t})\\
&=\frac{1}{(\sqrt{a^2+4 b})^{r-1}}\sum_{j=0}^{\frac{r-1}{2}}(-1)^j\binom{r}{j}\sum_{i=0}^\infty\frac{(a j t)^i}{i!}\sum_{k=0}^\infty u_k\frac{\bigl((r-2 j)t\bigr)^k}{k!}\\
&=\frac{1}{(\sqrt{a^2+4 b})^{r-1}}\sum_{j=0}^{\frac{r-1}{2}}(-1)^j\binom{r}{j}\sum_{n=0}^\infty\sum_{k=0}^n\binom{n}{k}(a j)^{n-k}(r-2 j)^k u_k\frac{t^n}{n!}\,.
\end{align*}
Therefore, we get
\begin{multline*}
\sum_{k_1+\cdots+k_r=n\atop k_1,\dots,k_r\ge 1}\binom{n}{k_1,\dots,k_r}u_{k_1}\cdots u_{k_r}\\
=\frac{1}{(\sqrt{a^2+4 b})^{r-1}}\sum_{j=0}^{\frac{r-1}{2}}(-1)^j\binom{r}{j}\sum_{k=0}^n\binom{n}{k}(a j)^{n-k}(r-2 j)^k u_k\,.
\end{multline*}
Similarly,  we get
$$
\sum_{k_1+\cdots+k_r=n\atop k_1,\dots,k_r\ge 0}\binom{n}{k_1,\dots,k_r}v_{k_1}\cdots v_{k_r}
=\sum_{j=0}^{\frac{r-1}{2}}\binom{r}{j}\sum_{k=0}^n\binom{n}{k}(a j)^{n-k}(r-2 j)^k v_k\,.
$$
If $r$ is even, then
\begin{align*}
u(t)^r
&=\frac{1}{(\sqrt{a^2+4 b})^r}\left(\sum_{j=0}^{\frac{r}{2}-1}(-1)^j\binom{r}{j}(e^{((r-j)\alpha+j\beta)t}+e^{(j\alpha+(r-j)\beta)t})\right.\\
&\quad \left.+(-1)^{\frac{r}{2}}\binom{r}{\frac{r}{2}}e^{(\frac{r}{2}\alpha+\frac{r}{2}\beta)t}\right)\\
&=\frac{1}{(\sqrt{a^2+4 b})^r}\left(\sum_{j=0}^{\frac{r}{2}-1}(-1)^j\binom{r}{j}e^{a j t}(e^{(r-2j)\alpha t}+e^{(r-2j)\beta t})\right.\\
&\quad \left.+(-1)^{\frac{r}{2}}\binom{r}{\frac{r}{2}}e^{a r t/2}\right)\\
&=\frac{1}{(\sqrt{a^2+4 b})^{r}}\left(\sum_{j=0}^{\frac{r}{2}-1}(-1)^j\binom{r}{j}\sum_{i=0}^\infty\frac{(a j t)^i}{i!}\cdot \sum_{k=0}^\infty v_k\frac{\bigl((r-2 j)t\bigr)^k}{k!}\right.\\
&\quad\left. +(-1)^{\frac{r}{2}}\binom{r}{\frac{r}{2}}\sum_{n=0}^\infty\left(\frac{a r}{2}\right)^n\frac{t^n}{n!}\right)\\
&=\frac{1}{(\sqrt{a^2+4 b})^{r}}\left(\sum_{j=0}^{\frac{r}{2}-1}(-1)^j\binom{r}{j}\sum_{n=0}^\infty\sum_{k=0}^n\binom{n}{k}(a j)^{n-k}(r-2 j)^k v_k\frac{t^n}{n!}\right.\\
&\quad\left. +(-1)^{\frac{r}{2}}\binom{r}{\frac{r}{2}}\sum_{n=0}^\infty\left(\frac{a r}{2}\right)^n\frac{t^n}{n!}\right)\,.
\end{align*}
Therefore, we get
\begin{multline*}
\sum_{k_1+\cdots+k_r=n\atop k_1,\dots,k_r\ge 1}\binom{n}{k_1,\dots,k_r}u_{k_1}\cdots u_{k_r}\\
=\frac{1}{(\sqrt{a^2+4 b})^{r}}\left(\sum_{j=0}^{\frac{r}{2}-1}(-1)^j\binom{r}{j}\sum_{k=0}^n\binom{n}{k}(a j)^{n-k}(r-2 j)^k v_k\right.\\
\quad \left.
+(-1)^{\frac{r}{2}}\binom{r}{\frac{r}{2}}\left(\frac{a r}{2}\right)^n\right)\,.
\end{multline*}
Similarly,  we get
\begin{multline*}
\sum_{k_1+\cdots+k_r=n\atop k_1,\dots,k_r\ge 0}\binom{n}{k_1,\dots,k_r}v_{k_1}\cdots v_{k_r}\\
=\sum_{j=0}^{\frac{r}{2}-1}\binom{r}{j}\sum_{k=0}^n\binom{n}{k}(a j)^{n-k}(r-2 j)^k v_k
+\binom{r}{\frac{r}{2}}\left(\frac{a r}{2}\right)^n\,.
\end{multline*}
\end{proof}

\section{Examples}

When $a=6$ and $b=-1$, then $B_n=u_n$ are balancing numbers and $C_n=v_n/2$ are Lucas-balancing numbers.  Thus, Theorem \ref{th-uv} can be reduced as follows.

\begin{Cor}
If $r$ is odd, then
\begin{multline*}
\sum_{k_1+\cdots+k_r=n\atop k_1,\dots,k_r\ge 1}\binom{n}{k_1,\dots,k_r}B_{k_1}\cdots B_{k_r}\\
=\frac{1}{(4\sqrt{2})^{r-1}}\sum_{k=0}^n\binom{n}{k}\sum_{j=0}^{\frac{r-1}{2}}(-1)^j\binom{r}{j}(6 j)^{n-k}(r-2 j)^k B_k
\end{multline*}
and
\begin{multline*}
\sum_{k_1+\cdots+k_r=n\atop k_1,\dots,k_r\ge 0}\binom{n}{k_1,\dots,k_r}C_{k_1}\cdots C_{k_r}\\
=\frac{1}{2^{r-1}}\sum_{k=0}^n\binom{n}{k}\sum_{j=0}^{\frac{r-1}{2}}\binom{r}{j}(6 j)^{n-k}(r-2 j)^k C_k\,.
\end{multline*}
If $r$ is even, then
\begin{multline*}
\sum_{k_1+\cdots+k_r=n\atop k_1,\dots,k_r\ge 1}\binom{n}{k_1,\dots,k_r}B_{k_1}\cdots B_{k_r}\\
=\frac{1}{(4\sqrt{2})^{r}}\left(2\sum_{k=0}^n\binom{n}{k}\sum_{j=0}^{\frac{r}{2}-1}(-1)^j\binom{r}{j}(6 j)^{n-k}(r-2 j)^k C_k\right.\\
\quad \left.
+(-1)^{\frac{r}{2}}\binom{r}{\frac{r}{2}}\left(\frac{r}{2}\right)^n\right)
\end{multline*}
and
\begin{multline*}
\sum_{k_1+\cdots+k_r=n\atop k_1,\dots,k_r\ge 0}\binom{n}{k_1,\dots,k_r}C_{k_1}\cdots C_{k_r}\\
=\frac{1}{2^r}\left(2\sum_{k=0}^n\binom{n}{k}\sum_{j=0}^{\frac{r}{2}-1}\binom{r}{j}(6 j)^{n-k}(r-2 j)^k C_k
+\binom{r}{\frac{r}{2}}\left(\frac{6 r}{2}\right)^n\right)\,.
\end{multline*}
\end{Cor}

When $a=b=1$, then $F_n=u_n$ are Fibonacci numbers and $L_n=v_n$ are Lucas numbers. Thus, Theorem \ref{th-uv} can be reduced as follows.

\begin{Cor}
If $r$ is odd, then
\begin{multline*}
\sum_{k_1+\cdots+k_r=n\atop k_1,\dots,k_r\ge 1}\binom{n}{k_1,\dots,k_r}F_{k_1}\cdots F_{k_r}\\
=\frac{1}{(\sqrt{5})^{r-1}}\sum_{k=0}^n\binom{n}{k}\sum_{j=0}^{\frac{r-1}{2}}(-1)^j\binom{r}{j}j^{n-k}(r-2 j)^k F_k
\end{multline*}
and
\begin{multline*}
\sum_{k_1+\cdots+k_r=n\atop k_1,\dots,k_r\ge 0}\binom{n}{k_1,\dots,k_r}L_{k_1}\cdots L_{k_r}\\
=\sum_{k=0}^n\binom{n}{k}\sum_{j=0}^{\frac{r-1}{2}}\binom{r}{j}j^{n-k}(r-2 j)^k L_k\,.
\end{multline*}
If $r$ is even, then
\begin{multline*}
\sum_{k_1+\cdots+k_r=n\atop k_1,\dots,k_r\ge 1}\binom{n}{k_1,\dots,k_r}F_{k_1}\cdots F_{k_r}\\
=\frac{1}{(\sqrt{5})^{r}}\left(\sum_{k=0}^n\binom{n}{k}\sum_{j=0}^{\frac{r}{2}-1}(-1)^j\binom{r}{j}j^{n-k}(r-2 j)^k L_k\right.\\
\quad \left.
+(-1)^{\frac{r}{2}}\binom{r}{\frac{r}{2}}\left(\frac{r}{2}\right)^n\right)
\end{multline*}
and
$$
\sum_{k_1+\cdots+k_r=n\atop k_1,\dots,k_r\ge 0}\binom{n}{k_1,\dots,k_r}L_{k_1}\cdots L_{k_r}
=\sum_{k=0}^n\binom{n}{k}\sum_{j=0}^{\frac{r}{2}-1}\binom{r}{j}j^{n-k}(r-2 j)^k L_k
+\binom{r}{\frac{r}{2}}\left(\frac{r}{2}\right)^n\,.
$$
\end{Cor}

When $r=2$ and $a=b=1$, we have
$$
\sum_{n=0}^n\binom{n}{k}F_k F_{n-k}=\frac{2^n L_n-2}{5}
$$
and
$$
\sum_{n=0}^n\binom{n}{k}L_k L_{n-k}=2^n L_n+2\,.
$$



\begin{thebibliography}{99}

\bibitem{AD1}
T. Agoh and K. Dilcher, {\em
Convolution identities and lacunary recurrences for Bernoulli numbers},
J. Number Theory 124 (2007), 105--122.

\bibitem{AD2}
T. Agoh and K. Dilcher, {\em
Higher-order recurrences for Bernoulli numbers},
J. Number Theory {\bf 129} (2009), 1837--1847.

\bibitem{AD3}
T. Agoh and K. Dilcher, {\em
Higher-order convolutions for Bernoulli and Euler polynomials},
J. Math. Anal. Appl. {\bf 419} (2014), 1235--1247.

\bibitem{behera-panda}
A. Behera and G. K. Panda, {\em On the square roots of triangular numbers},
Fibonacci Quart. 37 (1999) 98--105.

\bibitem{Finkelstein}
R. Finkelstein, {\em The house problem},
Amer. Math. Monthly {\bf 72} (1965), 1082--1088.

\bibitem{Komatsu2015}
T. Komatsu,  {\em
Higher-order convolution identities for Cauchy numbers of the second kind},
Proc. Jangjeon Math. Soc. {\bf 18} (2015), 369--383.

\bibitem{Komatsu2016}
T. Komatsu,  {\em Higher-order convolution identities for Cauchy numbers},
Tokyo J. Math. {\bf 39} (2016). 15 pages.

\bibitem{KMP}
T. Komatsu, Z. Masakova and E. Pelantova, {\em
Higher-order identities for Fibonacci numbers},
Fibonacci Quart. {\bf 52}, no.5 (2014), 150--163.

\bibitem{KS2014}
T. Komatsu and L. Szalay, {\em Balancing with binomial coefficients},
Intern. J. Number Theory {\bf 10} (2014), 1729--1742.

\bibitem{KS2016}
T. Komatsu and Y. Simsek,  {\em
Third and higher order convolution identities for Cauchy numbers},
Filomat {\bf 30} (2016), 1053--1060.

\bibitem{Liptai2004}
K. Liptai,  {\em Fibonacci Balancing numbers},
Fibonacci Quart. {\bf 42} (2004), 330--340.

\bibitem{LLPS}
K. Liptai, F. Luca, \'A. Pint\'er and L. Szalay, {\em Generalized balancing numbers},
Indag. Math. (N.S.) {\bf 20} (2009), 87--100.

\bibitem{panda2009}
G. K. Panda,  {\em Some fascinating properties of balancing numbers},  In Proc. of Eleventh Internat. Conference on Fibonacci Numbers and Their Applications, Cong. Numerantium {\bf 194} (2009), 185--189.

\bibitem{PR}
B. K. Patel and P. K. Ray,
{\em The period, rank and order of the sequence of balancing numbers modulo $m$}, Math. Rep. (Bucur.) \textbf{18}, No.3 (2016), Article No.9.

\bibitem{ray2012}
P. K. Ray, {\em Some congruences for balancing and Lucas-balancing numbers and their applications},
Integers, {\bf 14} (2014), \#A8.
\end{thebibliography}
\end{document}